\documentclass[a4paper]{amsart}
\usepackage{amsmath}
\usepackage{hyperref}
\usepackage{graphicx}
\usepackage{xcolor}

\newcommand{\abs}[1]{\lvert #1 \rvert}
\newcommand{\labs}[1]{\left| #1 \right|}
\newcommand{\scalarh}[2]{\langle #1 \,|\, #2 \rangle}
\newcommand{\scalarl}[2]{( #1 \,|\, #2 )}
\newcommand{\scalar}[2]{( #1 , #2 )}
\newcommand{\normh}[1]{\lVert #1 \rVert}
\newcommand{\norml}[2][2]{\labs{#2}_{#1}}
\newcommand{\settc}[2]{\bigl\{\,#1 \bigm\vert #2\,\bigr\}}
\newcommand{\tow}{\rightharpoonup}

\newcommand{\malf}{\boldsymbol{\alpha}}
\newcommand{\msig}{\boldsymbol{\Sigma}}
\newcommand{\alp}{\malf \cdot p}
\newcommand{\matr}[1]{\mathsf{#1}}
\DeclareMathOperator{\esp}{e}
\newcommand{\J}{J_{w}}
\newcommand{\dJ}[1]{dJ_{w}(\eta)[#1]}
\newcommand{\aeta}{a(\eta)}
\newcommand{\daeta}[1]{da(\eta)[#1]}
\newcommand{\acc}{\kappa}

\numberwithin{equation}{section}

\newtheorem{thm}[equation]{Theorem}
\newtheorem{lem}[equation]{Lemma}
\newtheorem{prop}[equation]{Proposition}

\theoremstyle{remark}
\newtheorem{rem}[equation]{Remark}

\title{Normalized solutions for the Klein Gordon-Dirac system}

\author[Coti Zelati]{Vittorio Coti Zelati}

\email[Coti Zelati]{vittorio.cotizelati@unina.it}

\address[Coti Zelati]{Dipartimento di Matematica Pura e Applicata 
``R.~Caccioppoli''\\
Universit\`a di Napoli ``Federico II''\\
via Cintia, M.S.~Angelo\\
80126 Napoli (NA), Italy}

\author{Margherita Nolasco} 

\email[Nolasco]{nolasco@univaq.it}

\address[Nolasco]{Dipartimento di Ingegneria e Scienze
dell'informazione e Matematica\\
Universit\`a dell'Aquila\\
via Vetoio, Loc.~Coppito\\
67010 L'Aquila (AQ) Italia}

\dedicatory{Dedicated to Antonio Ambrosetti, maestro and friend}

\begin{document}

\begin{abstract}
	We prove the existence of a stationary solution for the system
	describing the interaction between an electron coupled with a
	massless scalar field (a photon). We find a solution, with fixed
	$L^{2}$-norm, by variational methods, as a critical point of an
	energy functional.
\end{abstract}

\maketitle

\section{Introduction}
We study the interaction electron-photon analyizing the Euler-Lagrange
equations for a system consisting of a spinor field coupled with a
massless scalar field. More precisely our system consists of the Dirac
equation coupled with a massless Klein-Gordon equation, and look for
normalized and stationary solutions of the following system
\begin{equation}
	\label{eq:equazioneCompleta}
	\begin{cases}
		(-i\gamma^{\mu} \partial_{\mu} + m - s\varphi) \psi = 0 &
		\text{in } \mathbb{R} \times \mathbb{R}^{3} \\
		\partial^{\mu} \partial_{\mu} \varphi = 4\pi s
		\scalar{\psi}{\beta \psi}& \text{in } \mathbb{R} \times
		\mathbb{R}^{3} 
	\end{cases},
\end{equation}
where $\psi \colon \mathbb{R} \times \mathbb{R}^{3} \to
\mathbb{C}^{4}$, $\varphi \colon \mathbb{R} \times \mathbb{R}^{3} \to
\mathbb{R}$, $m > 0$ is the mass of the electron, $s > 0$ is the
coupling constant, $\gamma^{\mu}$ are the $4 \times 4$ Dirac matrices
\begin{equation*}
	\gamma^{0} = \beta = 
	\begin{pmatrix}	
		I & 0 \\
		0 & -I
	\end{pmatrix}
	\qquad
	\gamma^{k} = 
	\begin{pmatrix}	
		0 & \sigma^{k} \\
		-\sigma^{k} & 0
	\end{pmatrix}, \quad k = 1, \ldots, 3
\end{equation*}
and $\sigma^{k}$ are the $2 \times 2$ Pauli matrices
\begin{equation*}
	\sigma^{1} = 	
	\begin{pmatrix}	
		0 & 1 \\
		-1 & 0
	\end{pmatrix},
	\qquad
	\sigma^{2} = 	
	\begin{pmatrix}	
		0 & -i \\
		i & 0
	\end{pmatrix},
	\quad
	\sigma^{3} = 	
	\begin{pmatrix}	
		1 & 0 \\
		0 & -1
	\end{pmatrix},
\end{equation*}
and $\scalar{z}{w} = \sum_{i=1}^{4} \bar{z}_{i} w_{i}$ the scalar
product between $z$, $w \in \mathbb{C}^{4}$.

This problem is closely related to the one studied in
\cite{Nolasco_2021}, and we will prove existence of a solution of
\eqref{eq:equazioneCompleta} with essentially the same methods
developed in that article (see also \cite{CotiZelati_Nolasco_2019}).

More precisely we prove existence of stationary, normalized solutions
of this system, that is solutions $(\omega, \psi)$ of the following
problem:
\begin{equation}
	\label{eq:sistema}
	\begin{cases}
		(-i \malf \cdot \nabla + m\beta - s \varphi \beta) \psi = \omega
		\psi & \text{in } \mathbb{R}^{3} \\
		-\Delta \varphi = 4\pi \scalar{\psi}{\beta \psi}&
		\text{in } \mathbb{R}^{3} \\
		\norml{\psi}^{2} = \int_{\mathbb{R}^{3}} \abs{\psi(t, x)}^{2}
		\, dx = 1 
	\end{cases},
\end{equation}
where $\alpha_{i} = \beta \gamma_{i}$, $i = 1, \ldots, 3$. From
$-\Delta \varphi = 4\pi \scalar{\psi}{\beta \psi}$ we deduce that
\begin{equation*}
	\varphi = \scalar{\psi}{\beta \psi} * \frac{1}{\abs{x}}
\end{equation*}
and hence our problem reduces to
\begin{equation}
	\label{eq:sistemar}
	\begin{cases}
		(-i \malf \cdot \nabla + m\beta - s \scalar{\psi}{\beta \psi}
		* \frac{1}{\abs{x}} \beta) \psi = \omega \psi & \text{in }
		\mathbb{R}^{3} \\
		\norml{\psi}^{2} = \int_{\mathbb{R}^{3}} \abs{\psi(t, x)}^{2}
		\, dx = 1 
	\end{cases}.
\end{equation}
Our result is
\begin{thm}
	\label{thm:main}
	For all $s \in (0,\frac{1}{8\pi})$, there exists $\omega \in
	(0,m)$ and $\psi \in H^{1/2}(\mathbb{R}^{3}, \mathbb{C}^{4})$
	solutions of problem \eqref{eq:sistemar}.
\end{thm}
In the article \cite{Esteban_Georgiev_Sere_1996} the authors prove
using critical point theory the existence of one stationary solution
of equation \eqref{eq:equazioneCompleta} but do not prescribe its
$L^{2}$-norm.

We will find such a solution as a critical point of the functional 
\begin{equation*}
	I(\psi) = \frac{1}{2} \int_{\mathbb{R}^{3}} \scalar{H\psi}{\psi} -
	\frac{s}{4} \int_{\mathbb{R}^{3} \times \mathbb{R}^{3}}
	\frac{\scalar{\psi}{\beta \psi}(x) \scalar{\psi}{\beta
	\psi}(y)}{\abs{x - y}}
\end{equation*}
restricted on the manifold $\norml{\psi}^{2} = 1$.
Here 
\begin{equation*}
	H = -i \malf \cdot \nabla + m\beta.
\end{equation*}
The functional $I$ is strongly indefinite, and, following the method
introduced in \cite{CotiZelati_Nolasco_2019, Nolasco_2021}, the solution will
be found via a $\min$-$\max$ procedure consisting in minimizing the
supremum of $I$ over subspaces of dimension $1$ in the positive energy
subspace of the linear operator $H$, see proposition
\ref{prop:punticriticiE}. Let us remark here that we know very few
results on the existence of \emph{normalized} solutions for Dirac
equation (and more generally for strongly indefinite problems -- one
of these is \cite{Buffoni_Esteban_Sere_2006}).

\subsection{Notation and background results}

We let $\abs{u}^{p}_{p} = \int_{\mathbb{R}^{3}} \abs{u(x)}^{p}$, 
$\scalarl{u}{v} = \int_{\mathbb{R}^{3}} u(x) v(x)$.

Let us recall some well known facts on the Dirac operator $H$ (see for
more details \cite{Thaller_1992}): $H$ is a first order, self-adjoint
operator on $H^{1}(\mathbb{R}^{3}, \mathbb{C}^{4})$ with purely
absolutely continuous spectrum given by
\begin{equation*}
	\sigma(H) = (-\infty,-m] \cup [m,+\infty).
\end{equation*}
The orthogonal projectors $\Lambda_{\pm}$ on the positive and negative
energies subspaces are such that
\begin{equation*}
	H\Lambda_{\pm} = \Lambda_{\pm} H = \pm \sqrt{-\Delta + m^{2}}
	\Lambda_{\pm} = \pm \Lambda_{\pm} \sqrt{-\Delta + m^{2}}
\end{equation*}
and hence
\begin{equation*}
	\int \scalar{\psi(x)}{H \psi (x)} \, dx = \norml{(-\Delta + 
	m^{2})^{1/4} \Lambda_{+}\psi}^{2} - \norml{(-\Delta + 
	m^{2})^{1/4} \Lambda_{-}\psi}^{2}.
\end{equation*}
We will denote $X = H^{1/2}(\mathbb{R}^{3}, \mathbb{C}^{4})$, $X_{\pm}
= \Lambda_{\pm} X$, $\Sigma = \settc{\psi \in X}{\norml{\psi} = 1}$
and $\Sigma_{\pm} = \settc{\psi \in X_{\pm}}{\norml{\psi} = 1}$.

We have also that
\begin{align*}
	&\hat{H} = \mathcal{F} H \mathcal{F}^{-1} = \alp + m\beta \\
	&U \hat{H} U^{-1} = \lambda(p) \beta \\
	&\hat{\Lambda}_{\pm} = \mathcal{F} \Lambda_{\pm} \mathcal{F}^{-1}
	= U^{-1}\left(\frac{\matr{I} \pm \beta}{2}\right) U =
	\frac{1}{2}\left(\matr{I} \pm \frac{m}{\lambda(p)} \beta \pm
	\frac{1}{\lambda(p)} \alp\right)
\end{align*}
where
\begin{align*}
	&\mathcal{F} \psi(p) = \hat{\psi}(p) \left(= \frac{1}{(2\pi)^{3/2}}
	\int_{\mathbb{R}^{3}} \esp^{-ipx} \psi(x) \, dx \quad \text{for all $v 
	\in \mathcal{S}(\mathbb{R}^{3})$}\right),\\
	&\lambda(p) = \sqrt{\abs{p}^{2} + m^{2}},\\
	&U = u_{+}(p)\matr{I} + u_{-}(p) \beta \frac{\alp}{\abs{p}}, \\
	&U^{-1} = u_{+}(p)\matr{I} - u_{-}(p) \beta \frac{\alp}{\abs{p}},
	\\
	&u_{\pm}(p) = \sqrt{\frac{1}{2} \left(1 \pm \frac{m}{\lambda(p)}
	\right)}.
\end{align*}
	
Let, for $\phi$ and $\psi \in H^{1/2}(\mathbb{R}^{3}, C^{4})$
\begin{equation*}
	\scalarh{\phi}{\psi} = \int \sqrt{\abs{p}^{2} + m^{2}} 
	\scalar{\hat{\phi}(p)}{\hat{\psi}(p)} \, dp
\end{equation*}
and
\begin{equation*}
	\normh{\psi}^{2} = \scalarh{\psi}{\psi}.
\end{equation*}
We have that
\begin{equation*}
	\scalarh{\Lambda_{+}\phi}{\Lambda_{-}\psi} =
	\scalarl{\Lambda_{+}\phi}{\Lambda_{-}\psi} = 0.
\end{equation*}

Let us recall that, since $\mathcal{F}\frac{1}{\abs{x}} =
\sqrt{\frac{2}{\pi}}\frac{1}{\abs{p}^{2}}$, for all $f \in
L^{1}(\mathbb{R}^{3}) \cap L^{3/2}(\mathbb{R}^{3})$
\begin{equation}
	\label{eq:convoluzione_positiva}
	\int \frac{f(x)\bar{f}(y)}{\abs{x - y}} = 4\pi \int
	\frac{\abs{\hat{f}(p)}}{\abs{p}^{2}} \geq 0 
\end{equation}
and that for all $\rho \in L^{1}(\mathbb{R}^{3})$, $\psi \in 
H^{1/2}(\mathbb{R}^{3}, \mathbb{C}^{4})$
\begin{equation}
	\label{eq:stima_convoluzione}
	\int \frac{\rho(x)\abs{\psi}^{2}(y)}{\abs{x - y}} \leq \acc
	\abs{\rho}_{1} \norml{(-\Delta)^{1/4}\psi}^{2} \leq \acc
	\abs{\rho}_{1} \normh{\psi}^{2}
\end{equation}
($\acc = \frac{\pi}{2}$) and also that
\begin{equation}
	\label{eq:convoluzione_tende_0}
	\int \frac{\abs{f_{n}}(x) \abs{g_{n}}(x) \abs{h_{n}}(y)
	\abs{v}(y)}{\abs{x-y}} \to 0
\end{equation}
when $f_{n}$, $g_{n}$. $h_{n}$ and $v \in H^{1/2}$, $f_{n}$, $g_{n}$
bounded, $h_{n} \tow 0$.

\section{Maximization}

Let $I \colon H^{1/2}(\mathbb{R}^{3}, \mathbb{C}^{4}) \to \mathbb{R}$
\begin{equation*}
	I(\psi) = \frac{1}{2} \normh{\Lambda_{+}\psi}^{2} - \frac{1}{2}
	\normh{\Lambda_{-}\psi}^{2} - \frac{s}{4} \int
	\frac{\scalar{\psi}{\beta\psi}(x)
	\scalar{\psi}{\beta\psi}(y)}{\abs{x-y}}.
\end{equation*}

Let us fix $w \in \Sigma_{+}$ and let
\begin{equation*}
	B_{1} = \settc{\eta \in X_{-}}{\norml{\eta} < 1}.
\end{equation*}
We will look, given $w$, for a maximizer of the functional $\J$ 
defined on $B_{1}$
\begin{equation*}
	\J(\eta) = I(\aeta w + \eta) = \frac{1}{2} \normh{\aeta w}^{2} -
	\frac{1}{2} \normh{\eta}^{2} - \frac{s}{4} \int
	\frac{\scalar{\psi}{\beta\psi}(x)
	\scalar{\psi}{\beta\psi}(y)}{\abs{x-y}}
\end{equation*}
where $\aeta = \sqrt{1-\norml{\eta}^{2}}$ and $\psi = \aeta w + 
\eta \in \Sigma$. 

We have that $\daeta{\xi} = -\aeta^{-1} \scalarl{\eta}{\xi}$ and hence
the derivative of $\J$ is given, for all $\xi \in X_{-}$, by
\begin{multline}
	\label{eq:derivataJ}
	\dJ{\xi} = dI(\aeta w + \xi)[\daeta{\xi}w + \xi] = 
	dI(\psi)[h] \\
	= \scalarh{\aeta w}{\daeta{\xi} w} - \scalarh{\eta}{\xi} - s \int
	\frac{\scalar{\psi}{\beta\psi}(x) \scalar{\psi}{\beta h}(y)}
	{\abs{x-y}} \\
	= -\scalarl{\eta}{\xi} \normh{w}^{2} - \scalarh{\eta}{\xi} - s \int
	\frac{\scalar{\psi}{\beta\psi}(x) \scalar{\psi}{\beta h}(y)}
	{\abs{x-y}}
\end{multline}
(here $h = \daeta{\xi}w + \xi$) and we have, in particular
\begin{equation*}
	\dJ{\eta} = -\norml{\eta}^{2} \normh{w}^{2} - \normh{\eta}^{2} - s
	\int \frac{\scalar{\psi}{\beta\psi}(x) \scalar{\psi}{\beta
	(\daeta{\eta} w + \eta)}(y)} {\abs{x-y}}.
\end{equation*}

\begin{lem}
	\label{lem:stime_grad}
	For all $w \in \Sigma_{+}$ and $\eta \in B_{1}$ we have
	\begin{equation}
		\label{eq:bdd}
		\normh{\eta}^{2} \leq \aeta^{2} \normh{w}^{2} - 2\J(\eta),
	\end{equation}
	and for all $\eta \in B_{1}$ such that $\norml{\eta}^{2} \geq
	\frac{1}{2}$ and $\J(\eta) \geq 0$ we have that
	\begin{equation}
		\label{eq:spinge}
		\dJ{\eta} \leq - \frac{1}{2}(1 - 4s\acc )m < 0
	\end{equation}
\end{lem}

\begin{proof}	
	We have, thanks to \eqref{eq:convoluzione_positiva} that for 
	$\eta \in B_{1}$ and $\psi = \aeta w + \eta$
	\begin{equation*}
		\frac{1}{2} \normh{\eta}^{2} \leq \frac{1}{2}
		\normh{\eta}^{2} + \frac{s}{4} \int
		\frac{\scalar{\psi}{\beta\psi}(x)
		\scalar{\psi}{\beta\psi}(y)}{\abs{x-y}} 
		= \frac{1}{2} \normh{a(\eta)w}^{2} - \J(\eta)
	\end{equation*}
	and \eqref{eq:bdd} follows.
	
	From \eqref{eq:bdd} follows  that $\normh{\eta} \leq \aeta 
	\normh{w}$ if $\J(\eta) \geq 0$, hence we have, if 
	$\norml{\eta}^{2} > \frac{1}{2}$:
	\begin{align*}
		\dJ{\eta} &= -\norml{\eta}^{2} \normh{w}^{2} -
		\normh{\eta}^{2} - s \int \frac{\scalar{\psi}{\beta\psi}(x)
		\scalar{\psi}{\beta \psi}(y)} {\abs{x-y}} \\
		&\quad + s \int \frac{\scalar{\psi}{\beta\psi}(x)
		\scalar{w}{\beta w}(y)} {\abs{x-y}} \\
		&\quad + s \aeta^{-1} \int \frac{\scalar{\psi}{\beta\psi}(x)
		\scalar{\eta}{\beta w}(y)} {\abs{x-y}} \\
		&\leq-\norml{\eta}^{2} \normh{w}^{2} - \normh{\eta}^{2} + s
		\acc (\normh{w}^{2} + \aeta^{-1}\normh{\eta} \normh{w}) \\
		&\leq- \frac{1}{2} \normh{w}^{2} - \normh{\eta}^{2} + 2s \acc
		\normh{w}^{2} < - \frac{1}{2}\left(1 - 4s\acc
		\right)\normh{w}^{2} \\
		&\leq - \frac{1}{2}\left(1 - 4s\acc \right)m\norml{w}^{2} = -
		\frac{1}{2}\left(1 - 4s\acc \right)m
	\end{align*}
	where we have used \eqref{eq:convoluzione_positiva} and
	\eqref{eq:stima_convoluzione}.
\end{proof}

\begin{rem}
	\label{rem:psbdd}
	It follows from the Lemma \ref{lem:stime_grad} that if $\eta_{n}$
	is a Palais-Smale sequence for $\J$ such that $\J(\eta_{n}) \geq
	0$, then $\norml{\eta_{n}}^{2} < \frac{1}{2}$ for all $n \in
	\mathbb{N}$ large enough.
\end{rem}

\begin{lem}
	\label{lem:PSvale}
	Let $\eta_{n} \in B_{1}$ be a Palais-Smale sequence for $\J$, 
	that is 
	\begin{equation*}
		\J(\eta_{n}) \to c \geq 0, \qquad d\J(\eta_{n}) \to 0.
	\end{equation*}
	
	Then $\eta_{n}$ converges, up to a subsequence, to a critical 
	point $\eta$ of $\J$.
\end{lem}

\begin{proof}
	Follows form the Lemma \ref{lem:stime_grad} and Remark
	\ref{rem:psbdd} that $\norml{\eta_{n}}^{2} < \frac{1}{2}$ and that
	$\normh{\eta_{n}}$ is bounded, hence $\eta_{n} \tow \eta$ (up to a
	subsequence).
	
	From
	\begin{multline*}
		o(1) = d\J(\eta_{n})[\eta_{n} - \eta] = -
		\scalarl{\eta_{n}}{\eta_{n} - \eta} \normh{w}^{2} -
		\scalarh{\eta_{n}}{\eta_{n} - \eta} \\
		- s \int \frac{\scalar{\psi_{n}}{\beta \psi_{n}}(x)
		\scalar{\psi_{n}}{\beta(-a(\eta_{n})^{-1}
		\scalarl{\eta_{n}}{\eta_{n} - \eta}w + \eta_{n}
		-\eta)}(y)}{\abs{x - y}}
	\end{multline*}
	we deduce that
	\begin{multline*}
		\norml{\eta_{n} - \eta}^{2} \normh{w}^{2} + \normh{\eta_{n} -
		\eta}^{2} = - \scalarl{\eta}{\eta_{n} - \eta} \normh{w}^{2} -
		\scalarh{\eta}{\eta_{n} - \eta} \\
		+sa(\eta_{n})^{-1}\scalarl{\eta_{n}}{\eta_{n}-\eta}) \int
		\frac{\scalar{\psi_{n}}{\beta \psi_{n}}(x)
		\scalar{\psi_{n}}{\beta w}(y)}{\abs{x - y}} \\
		- s \int \frac{\scalar{\psi_{n}}{\beta \psi_{n}}(x)
		\scalar{\psi_{n}}{\beta(\eta_{n} - \eta)}(y)}{\abs{x - y}} + 
		o(1).
	\end{multline*}
	We have that 
	\begin{align*}
		\int \frac{\scalar{\psi_{n}}{\beta \psi_{n}}(x)
		\scalar{\psi_{n}}{\beta(\eta_{n} - \eta)}(y)}{\abs{x - y}} 
		&= \int \frac{\scalar{\psi_{n}}{\beta (\eta_{n} -\eta)}(x)
		\scalar{\psi_{n}}{\beta(\eta_{n} - \eta)}(y)}{\abs{x - y}} \\
		&\qquad + \int \frac{\scalar{\psi_{n}}{\beta \eta}(x)
		\scalar{\psi_{n}}{\beta(\eta_{n} - \eta)}(y)}{\abs{x - y}} \\
		&\qquad + a(\eta_{n}) \int \frac{\scalar{\psi_{n}}{\beta w}(x)
		\scalar{\psi_{n}}{\beta(\eta_{n} - \eta)}(y)}{\abs{x - y}} \\
		&= \int \frac{\scalar{\psi_{n}}{\beta (\eta_{n} - \eta)}(x)
		\scalar{\psi_{n}}{\beta(\eta_{n} - \eta)}(y)}{\abs{x - y}} +
		o(1)
	\end{align*}
	and
	\begin{align*}
		\labs{\int \frac{\scalar{\psi_{n}}{\beta \psi_{n}}(x)
		\scalar{\psi_{n}}{\beta a(\eta_{n}) w}(y)}{\abs{x - y}}} &
		\leq \acc \normh{\psi_{n}} \normh{a(\eta_{n}) w} \\
		&\leq \acc (2 \normh{a(\eta_{n}) w}^{2} +
		\normh{\eta_{n}}^{2} ) \leq 3 \acc a(\eta_{n})^{2}
		\normh{w}^{2}.
	\end{align*}
	Since $\norml{\eta_{n}}^{2} < \frac{1}{2}$ we have
	\begin{equation*}
		a(\eta_{n})^{-2} \scalarl{\eta_{n}}{\eta_{n} - \eta} =
		a(\eta_{n})^{-2} \norml{\eta_{n} - \eta}^{2} + o(1)
	\end{equation*}
	and we deduce that
	\begin{multline*}
		\norml{\eta_{n} - \eta}^{2} \normh{w}^{2} + \normh{\eta_{n} -
		\eta}^{2} \leq 3s\acc \norml{\eta_{n} - \eta}^{2}
		\normh{w}^{2} \\
		- s\int \frac{\scalar{\psi_{n}}{\beta (\eta_{n} - \eta)}(x)
		\scalar{\psi_{n}}{\beta(\eta_{n} - \eta)}(y)}{\abs{x - y}} +
		o(1) \\
		\leq 3s\acc \norml{\eta_{n} - \eta}^{2}
		\normh{w}^{2} + o(1)
	\end{multline*}
	and $\eta_{n} \to \eta$, with $\eta$ critical point of $\J$.
\end{proof}

We now show that all the critical points of $\J$ at positive levels 
are strict local maxima. This lemma follows as in 
\cite{CotiZelati_Nolasco_2019, Nolasco_2021}
\begin{lem}
	\label{lem:criticomax}
	Let $\eta \in X_{-}$ a critical point of $\J$ such that $\J(\eta) \geq 
	0$.
	
	Then there exists $\delta > 0$ such that 
	\begin{equation*}
		d^{2}\J(\eta)[\xi, \xi] \leq -\delta \normh{\xi}^{2} \qquad
		\text{ for all } \xi \in X_{-}.
	\end{equation*}

\end{lem}

\begin{proof}
	In order to compute the second derivative se denote $\psi = \aeta
	w + \eta$ and $h = \daeta{\xi}w + \xi$ and observe that
	\begin{equation*}
		d^{2}\aeta[\xi,\zeta] = -\aeta^{-1} \left(
		\scalarl{\xi}{\zeta} + \frac{\scalarl{\eta}{\xi}
		\scalar{\eta}{\zeta}}{1 - \norml{\eta}^{2}} \right).
	\end{equation*}
	Then	
	\begin{align*}
		d^{2}\J(\eta)[\xi, &\xi] = d^{2}I(\psi)[h,
		h] + dI(\psi)[d^{2}\aeta[\xi, \xi]w] \\
		&= \scalarh{\daeta{\xi}w}{\daeta{\xi}w} - \scalarh{\xi}{\xi}
		-s \int \frac{\scalar{\psi}{\beta \psi}(x)\scalar{h}{\beta
		h}(y)}{\abs{x - y}} \\
		&\quad -2s\int \frac{\scalar{\psi}{\beta h}(x)\scalar{\psi}{\beta
		h}(y)}{\abs{x - y}} + \aeta d^{2}\aeta[\xi, \xi] 
		\scalarh{w}{w} \\
		&\quad -s\int \frac{\scalar{\psi}{\beta
		\psi}(x)\scalar{\psi}{\beta d^{2}\aeta[\xi.\xi] w}(y)}{\abs{x
		- y}} \\
		&= - \norml{\xi}^{2}\normh{w}^{2} + s\norml{\xi}^{2} \int
		\frac{\scalar{\psi}{\beta \psi}(x)\scalar{w}{\beta
		w}(y)}{\abs{x - y}} - \normh{\xi}^{2} \\
		&\quad -s\int \frac{\scalar{\psi}{\beta \psi}(x)\scalar{\xi}{\beta
		\xi}(y)}{\abs{x - y}} + 2s \frac{\scalarl{\eta}{\xi}}{1 - 
		\norml{\eta}^{2}} \aeta \Gamma(\xi) \\
		&\quad -2s\int \frac{\scalar{\psi}{\beta
		h}(x)\scalar{\psi}{\beta h}(y)}{\abs{x - y}} + s
		\left(\frac{\norml{\xi}^{2}}{1- \norml{\eta}^{2}} +
		\frac{\scalarl{\eta}{\xi}^{2}}{(1 - \norml{\eta}^{2})^{2}}
		\right) \aeta \Gamma(\eta)
	\end{align*}
	where we have set
	\begin{equation*}
		\Gamma(\zeta) = \int \frac{\scalar{\psi}{\beta
		\psi}(x)\scalar{w}{\beta \zeta}(y)}{\abs{x - y}}.
	\end{equation*}
	Since $\eta$ is a critical point for $\J$, we have that
	\begin{multline*}
		d^{2}\J(\eta)[\xi, \xi] = d^{2}\J(\eta)[\xi, \xi] + 2
		\frac{\scalarl{\eta}{\xi}}{1 - \norml{\eta}^{2}} \dJ{\xi} \\
		+ \left( \frac{\norml{\xi}^{2}}{1 - \norml{\eta}^{2}} + 3
		\frac{\scalar{\eta}{\xi}^{2}}{(1 - \norml{\eta}^{2})^{2}}
		\right) \dJ{\eta}.
	\end{multline*}
	We have that
	\begin{align*}
		0 &= \dJ{\eta} = -\norml{\eta}^{2} \normh{w}^{2} -
		\normh{\eta}^{2} + s \frac{\norml{\eta}^{2}}{1 -
		\norml{\eta}^{2}} \int \frac{\scalar{\psi}{\beta\psi}(x)
		\scalar{w}{\beta w}(y)} {\abs{x-y}} \\
		& -s \int \frac{\scalar{\psi}{\beta\psi}(x)
		\scalar{\eta}{\beta \eta}(y)} {\abs{x-y}} - s \left(\frac{1 -
		2\norml{\eta}^{2}}{1 - \norml{\eta}^{2}}\right) a 
		\Gamma(\eta)\\
		&\leq -\norml{\eta}^{2}(1 - \acc)\normh{w} - (1 -
		\acc)\normh{\eta} - s \left(\frac{1 -
		2\norml{\eta}^{2}}{1 - \norml{\eta}^{2}}\right) a \Gamma(\eta)
	\end{align*}
	which implies that $\Gamma(\eta) < 0$. After some semplification
	we get
	\begin{align*}
		d^{2}\J(\eta)[\xi, \xi] &\leq -Q (1 -
		s\acc)\normh{w}^{2} - \frac{\norml{\xi}^{2}}{1 -
		\norml{\eta}^{2}} (1 - s\acc) \normh{\eta}^{2} \\
		&\quad - (1 - s\acc)\normh{\xi + R \eta}^{2} \\
		&\leq - \frac{\norml{\xi}^{2}}{1 - \norml{\eta}^{2}} (1 -
		s\acc) \normh{\eta}^{2} - (1 -
		s\acc)\left(\frac{1}{3} \normh{\xi}^{2} - \frac{1}{2}
		R^{2}\normh{\eta}^{2} \right) \\
		&\leq -\frac{1}{3} (1 - s\acc) \normh{\xi}^{2}
	\end{align*}
	where 
	\begin{align*}
		&R = \frac{\scalar{\eta}{\xi}}{1 - \norml{\eta}^{2}} \\
		&Q = \norml{\xi}^{2} + 2R\scalar{\eta}{\xi} - \norml{\eta}^{2}
		\left(\frac{\norml{\xi}^{2}}{1 - \norml{\eta}^{2}} + R^{2}\right).
	\end{align*}
	Remark that, since $\norml{\eta}^{2} < \frac{1}{2}$, we have
	\begin{equation*}
		Q \geq \left(\frac{\norml{\xi}}{1 - \norml{\eta}^{2}} +
		R^{2}\right)(1 - 2\norml{\eta}^{2}) > 0
	\end{equation*}
	and
	\begin{equation*}
		\frac{\norml{\xi}^{2}}{1 - \norml{\eta}^{2}} - \frac{1}{2}
		\frac{\scalarl{\eta}{\xi}^{2}}{(1 - \norml{\eta}^{2})^{2}}
		\geq \frac{\norml{\xi}^{2}(2 - 3\norml{\eta}^{2})}{2(1 -
		\norml{\eta}^{2})^{2}} > \frac{\norml{\xi}^{2}}{4(1 -
		\norml{\eta}^{2})^{2}} > 0.
	\end{equation*}
\end{proof}

We let, for all $w \in \Sigma_{+}$
\begin{equation*}
	\mathcal{E}(w) = \sup_{\eta \in B_{1}} \J(\eta).
\end{equation*}

\begin{lem}
	\label{lem:stime_max}
	For all $w \in \Sigma_{+}$ we have
	\begin{equation*}
		0 < \frac{1}{4}(2 - s\acc)m \leq \frac{1}{4}(2 -
		s\acc) \normh{w}^{2} \leq \mathcal{E}(w) \leq
		\frac{1}{2} \normh{w}^{2}.
	\end{equation*}
\end{lem}

\begin{proof}
	We have that 
	\begin{align*}
		\mathcal{E}(w) &\geq \J(0) = \frac{1}{2} \normh{w}^{2} -
		\frac{s}{4} \int \frac{\scalar{w}{\beta w}(x)\scalar{w}{\beta
		w}(y)}{\abs{x-y}} \\
		&\geq \frac{1}{4}(2 - s\acc) \normh{w}^{2} \geq
		\frac{1}{4}(2 - s\acc) m \norml{w}^{2} = \frac{1}{4}(2 -
		s\acc) m
	\end{align*}
	and, for all $\eta \in B_{1}$, we have
	\begin{equation*}
		\J(\eta) = \frac{1}{2} \normh{\aeta w}^{2} - \frac{1}{2}
		\normh{\eta}^{2} - \frac{s}{4} \int
		\frac{\scalar{\psi}{\beta\psi}(x)
		\scalar{\psi}{\beta\psi}(y)}{\abs{x-y}} \leq \frac{1}{2}
		\normh{w}^{2}.
	\end{equation*}
\end{proof}

\begin{prop}
	\label{prop:massimo}
	For every $w \in \Sigma_{+}$ there is a unique $\eta(w) \in B_{1}$
	such that
	\begin{equation*}
		\J(\eta(w)) = \max_{\eta \in B_{1}} \J(\eta) = \mathcal{E}(w).
	\end{equation*}
	
	 $\eta(w)$ is a critical point of $\J$ on $B_{1}$ such that
	 $\norml{\eta(w)} < \frac{1}{2}$ and
	\begin{equation}
		\label{eq:stimeeta}
		\normh{\eta(w)}^{2} + m \leq \normh{\aeta w}^{2}, \qquad
		\normh{\eta(w)}^{2} \leq \frac{s}{2}\acc \normh{w}^{2}.
	\end{equation}
	Moreover the map 
	\begin{equation*}
		w \in X_{+} \setminus \{0\} \mapsto \gamma(w) =
		\eta(\norml{w}^{-1}w) \in B_{1}
	\end{equation*} 
	is smooth.
\end{prop}
	
\begin{proof}
	We can find, by lemma \ref{lem:stime_max} and using Ekeland's
	variational principle a maximizing Palais-Smale sequence
	$\eta_{n}$ at positive level.
	
	Then, by lemma \ref{lem:PSvale}, $\eta_{n} \to \eta$ (up to a 
	subsequence), with 
	\begin{equation*}
		d\J(\eta) = 0, \qquad \J(\eta) = \mathcal{E}(w).
	\end{equation*}
	From
	\begin{multline*}
		\mathcal{E}(w) = \frac{1}{2}\aeta^{2}\normh{w}^{2} - \frac{1}{2}
		\normh{\eta}^{2} - \frac{s}{4} \int
		\frac{\scalar{\psi}{\beta\psi}(x)
		\scalar{\psi}{\beta\psi}(y)}{\abs{x-y}} \\
		\geq \frac{1}{2} \normh{w}^{2} - \frac{s}{4} \int
		\frac{\scalar{w}{\beta w}(x) \scalar{w}{\beta
		w}(y)}{\abs{x-y}}
	\end{multline*}
	we deduce, using Lemma \ref{lem:stimaMalefica} in the appendix:
	\begin{align*}
		\aeta^{2}\normh{w}^{2} - \normh{\eta}^{2} - m
		&\geq \normh{w}^{2} - s \acc \norml{\eta}^{2}
		\normh{w}^{2} \\
		&\qquad -7 s \aeta^{2} \acc(\normh{w}^{2} -
		m\norml{w}^{2}) - 9 s\acc\normh{\eta}^{2} - m\norml{w}^{2} \\
		&\geq 9s\acc(\aeta^{2}\normh{w}^{2} - \normh{\eta}^{2} -
		m\norml{w}^{2}) \\
		&\qquad + (1 - 16s\acc)(\normh{w}^{2} - m\norml{w}^{2})
		+ 8 s \acc \norml{\eta}^{2}\normh{w}^{2}
	\end{align*}
	and we immediately deduce that 
	\begin{equation*}
		\aeta^{2}\normh{w}^{2} - \normh{\eta}^{2} - m\norml{w}^{2}
		\geq \frac{1 - 16s\acc}{1 - 9s\acc}(\normh{w}^{2}
		- m\norml{w}^{2}) > 0.
	\end{equation*}
	We also have that
	\begin{equation*}
		\norml{\eta}^{2}\normh{w}^{2} + \normh{\eta}^{2} \leq
		\frac{s}{2} \int \frac{\scalar{w}{\beta w}(x) \scalar{w}{\beta
		w}(y)}{\abs{x-y}} \leq \frac{s}{2}\acc
		\norml{(-\Delta)^{-1/4}w}^{2} \leq \frac{s}{2}\acc
		\normh{w}^{2}.
	\end{equation*}
	from which we deduce
	\begin{equation*}
		\normh{\eta}^{2} \leq \frac{s}{2}\acc \normh{w}^{2}.
	\end{equation*}
	
	To prove the uniqueness of the maxima for $J_{w}(\eta)$ we 
	assume, by contradiction, the existence of $\eta_{1}$, $\eta_{2} 
	\in B_{1}$ such that
	\begin{equation*}
		\J(\eta_{1}) = \J(\eta_{2}) = \mathcal{E}(w).
	\end{equation*}
	Follows from lemma \ref{lem:stime_grad} that $\norml{\eta_{1}}^{2}
	< \frac{1}{2}$ and $\norml{\eta_{2}}^{2} < \frac{1}{2}$. We will
	use the mountain pass lemma in order to reach a contradiction. Let
	\begin{equation*}
		\Gamma = \settc{g \in C([0,1], B_{1})}{g(0) = \eta_{1}, \ g(1)
		= \eta_{2}, \ \norml{g(t)}^{2} < \frac{1}{2}}
	\end{equation*}
	and define the min-max level
	\begin{equation*}
		c = \sup_{g \in \Gamma} \min_{t \in [0,1]} J_{w}(g(t)).
	\end{equation*}
	
	Let $g(t) = t\eta_{1} + (1-t)\eta_{2}$. We have that
	$\norml{g(t)}^{2} < \frac{1}{2}$ and $a(g(t))^{2} > \frac{1}{2}$
	for all $t \in [0,1]$, so that se have, letting $\psi_{t} =
	a(g(t))w + g(t)$:
	\begin{align*}
		\J(g(t)) &= \frac{1}{2} a(g(t))^{2} \normh{w}^{2} - 
		\frac{1}{2} \normh{g(t)}^{2} - \frac{s}{4} \int 
		\frac{\scalar{\psi_{t}}{\beta \psi_{t}}(x) 
		\scalar{\psi_{t}}{\beta \psi_{t}}(y)}{\abs{x - y}} \\
		&\geq \frac{1}{2} (1 - \frac{s\acc}{2})a(g(t))^{2}
		\normh{w}^{2} - \frac{1}{2} (1 + \frac{s\acc}{2})
		\normh{g(t)}^{2} \\
		&\geq \frac{1}{2} (1 - \frac{s\acc}{2})ta(\eta_{1})^{2}
		\normh{w}^{2} + \frac{1}{2} (1 -
		\frac{s\acc}{2})(1-t)a(\eta_{2})^{2} \normh{w}^{2} \\
		&\qquad - \frac{1}{2} (1 + \frac{s\acc}{2}) t
		\normh{\eta_{1}}^{2} - \frac{1}{2} (1 + \frac{s\acc}{2})
		(1-t) \normh{\eta_{1}}^{2} \\
		&\geq\frac{1}{4} (1 - \frac{s\acc}{2}) \normh{w}^{2} -
		\frac{1}{2} (1 + \frac{s\acc}{2}) \frac{s}{2}\acc
		\normh{w}^{2} \geq\frac{1}{4} (1 - 2 s\acc)
		\normh{w}^{2}
	\end{align*}
	where we have used the second inequality in \eqref{eq:stimeeta}
	and $s\acc < 2$. We deduce that $c > 0$. Since the set
	$\Gamma$ is invariant for the gradient flow (see Lemma
	\ref{lem:stime_grad}) and the Palais-Smale condition holds (see
	lemma \ref{lem:PSvale}) we deduce from the mountain pass theorem
	that there exist a critical point at level $c$, which cannot be a
	strict local maximum. The contradiction then follows form lemma
	\ref{lem:criticomax}.

	To prove that the map $w \mapsto \gamma(w) =
	\eta(\norml{w}^{-1}w)$ is smooth we consider, since the map $w
	\mapsto P(w) = \norml{w}^{-1}w$ is smooth, $(w_{0}, \eta(w_{0}))
	\in \Sigma_{+} \times B_{1}$ and we let $V \subset X_{+} \setminus
	\{0\}$ and $U \subset B_{1}$ be neighborhoods of $w_{0}$ and
	$\eta(w_{0})$ respectively. Then we define the map $F \colon V
	\times U \to L(X_{-})$ as
	\begin{equation*}
		F(w,\eta)[\xi] = dJ_{P(w)}(\eta)[\xi] \qquad \xi \in X_{-}.
	\end{equation*}
	Clearly $P(w_{0}) = w_{0}$ and $F(w_{0},\eta(w_{0})) = 0$. We have
	that
	\begin{equation*}
		d_{\eta}F(w_{0},\eta(w_{0}))[\xi][\zeta] =
		d^{2}J_{w_{0}}(\eta(w_{0}))[\xi, \zeta], \qquad \xi, \zeta \in
		X_{-}.
	\end{equation*}
	Follows form lemma \ref{lem:criticomax} that 
	\begin{equation*}
		- d_{\eta}F(w_{0},\eta(w_{0}))[\xi][\xi] = -
		d^{2}J_{w_{0}}(\eta(w_{0}))[\xi, \xi] \geq \delta
		\normh{\xi}^{2} \qquad \text{for all } \xi \in X_{-}
	\end{equation*}
	and hence we have from Lax-Milgram that for all linear 
	functionals $L$ on $X_{-}$ there is a unique $\zeta \in X_{-}$ 
	such that 
	\begin{equation*}
		- d_{\eta}F(w_{0},\eta(w_{0}))[\zeta][\xi] = L[\xi], \qquad 
		\text{for all } \xi \in X_{-}
	\end{equation*}
	that is, $L = - d_{\eta}F(w_{0},\eta(w_{0}))[\zeta]$. By the
	Implicit Function theorem, there exist $V_{0} \subset V$ and
	$U_{0} \subset U$, neighborhoods of $w_{0}$ and $\eta(w_{0})$ and
	a smooth map $\gamma \colon V_{0} \to U_{0}$ such that $F(w,
	\gamma(w)) = 0$ for all $w \in V_{0}$, that is, $\gamma(w)$ is a
	critical point of $J_{P(w)}$ on $B_{1}$ at a positive level. Then,
	by Proposition \ref{lem:criticomax}, $\gamma(w)$ is a strict local
	maximum of $J_{P(w)}$ on $B_{1}$. Again using the Mountain Pass 
	theorem we deduce that actually $\gamma(w) = \eta(P(w))$ is the 
	unique (up to a phase factor) maximum of $J_{P(w)}$.
	
	Finally we have that
	\begin{equation*}
		d\gamma(w)[v] = -d_{\eta}F(w,\gamma(w))^{-1} [d_{w}F(w, 
		\gamma(w))[v]] \qquad \text{for all } v \in X_{+}.
	\end{equation*}
\end{proof}

Follows from proposition \ref{prop:massimo} that we can consider the 
smooth functional 	$\mathcal{E} \colon X_{+} \setminus \{0\} \to 
\mathbb{R}$ defined as
\begin{equation*}
	\mathcal{E}(w) = J_{P(w)}(\gamma(w)) = \sup_{\eta \in B_{1}} 
	J_{P(w)}(\eta).
\end{equation*}
Since 
\begin{equation*}
	J_{P(w)}(\gamma(w)) = I(a(\gamma(w))P(w) + \gamma(w))
\end{equation*}
and recalling that
\begin{equation*}
	dJ_{P(w)}(\gamma(w))[\xi] = dI(\psi_{w})[da(\gamma(w))[\xi]P(w) +
	\xi] = 0 \qquad \text{for all } \xi \in X_{-},
\end{equation*}
(were $\psi_{w} = a(\gamma(w))P(w) + \gamma(w)$) we have that for all
$v \in X_{+}$
\begin{align*}
	d\mathcal{E}(w)[v] &= d_{w}J_{P(w)}(\gamma(w))[v] \\
	&= dI(\psi_{w})[da(\gamma(w))[d\gamma(w)[v]]P(w) +
	a(\gamma(w))dP(w)[v] + d\gamma(w)[v]] \\
	&= d_{\eta}J_{P(w)}(\gamma(w))[d\gamma(w)[v]] +
	dI(\psi_{w})[a(\gamma(w))dP(w)[v]] \\
	&= dI(\psi_{w})[a(\gamma(w))dP(w)[v]] \\
	&= dI(\psi_{w})[a(\gamma(w))v] -
	dI(\psi_{w})[a(\gamma(w))\scalarl{w}{v}w] 
\end{align*}
(we have used that $dP(w)[v] = v - \scalarl{w}{v} w$) and
\begin{equation}
	\label{eq:derivataE}
	d\mathcal{E}(w)[v] = a(\gamma(w))dI(\psi_{w})[v] -
	a(\gamma(w))^{2} \omega(\psi_{w})\scalarl{w}{v} \qquad \text{for
	all } v \in X_{+}.
\end{equation}
where
\begin{equation}
	\label{eq:moltiplicatore}
	\omega(\psi_{w}) = a(\gamma(w))^{-1} dI(\psi_{w})[w]
\end{equation}

\begin{prop}
	\label{prop:punticriticiE}
	Let $w_{0} \in \Sigma_{+}$ be a critical points of $\mathcal{E}$
	restricted on the manifold $\Sigma_{+}$.
	Then $w_{0}$ is a critical points for $\mathcal{E}$ on $X_{+}$ and 
	the function
	\begin{equation*}
		\psi_{0} = a(\eta(w_{0})) w_{0} + \eta(w_{0}) \in \Sigma
	\end{equation*}
	is a critical point for $I$ on the manifold $\Sigma$ and satisfies
	\begin{equation}
		\label{eq:equazionemoltiplicatore}
		dI(\psi_{0})[h] = \omega \scalarl{\psi_{0}}{h}
		\qquad \text{for all } h \in X
	\end{equation}
	where $\omega = \omega(\psi_{0}) \in \mathbb{R}$,
	\begin{equation*}
		\label{eq:stimeMoltiplicatore}
		(1-3s\acc) \normh{w_{0}}^{2} \leq \omega(\psi_{0}) \leq 2
		I(\psi_{w_{0}}) = 2\mathcal{E}(w_{0}).
	\end{equation*}
	
	Moreover, if $\psi_{0} \in X$ satisfies $I(\psi_{0}) \geq 0$ and 
	\eqref{eq:equazionemoltiplicatore} for some $\omega \in
	\mathbb{R}$, then $w = \norml{\Lambda_{+}\psi_{0}}^{-1}
	\Lambda_{+} \psi_{0}$ is a critical point for $\mathcal{E}(w)$ .
\end{prop}

\begin{proof}
	Let $w_{0} \in \Sigma_{+}$ be a critical point for $\mathcal{E}$
	on $\Sigma_{+}$, $\eta_{0} = \eta(w_{0}) = \gamma(w_{0})$,
	$\psi_{0} = a(\eta_{0})w_{0} + \eta_{0}$. Then
	\begin{equation*}
		d\mathcal{E}(w_{0})[h] = 0 \qquad \text{for all } h \in 
		T_{w_{0}}\Sigma_{+} = \settc{h \in X_{+}}{\scalarl{w_{0}}{h} = 0}.
	\end{equation*}
	Since $dP(w_{0})[w_{0}] = 0$ we immediately deduce that 
	\begin{equation*}
		d\mathcal{E}(w_{0})[v] = 0 \qquad \text{for all } v \in X_{+}.
	\end{equation*}
	From
	\eqref{eq:derivataJ} we have that for all $\xi \in X_{-}$
	\begin{equation*}
		0= dJ_{w_{0}}(\eta_{0})[\xi] = dI(\psi_{0})[\xi] +
		dI(\psi_{0})[(da(\eta_{0})[\xi]) w_{0} ]
	\end{equation*}
	while for all $v \in X_{+}$ we have
	\begin{equation*}
		0 = d\mathcal{E}(w_{0})[v] = a(\eta_{0}) dI(\psi_{0})[v]
		- a(\eta_{0})^{2} \omega(\psi_{0}) \scalarl{w_{0}}{v}
	\end{equation*}
	and hence, for all $h = v + \xi$, $v \in X_{+}$, $\xi \in X_{-}$
	\begin{align*}
		dI(\psi_{0})[h] &= a(\eta_{0}) \omega(\psi_{0})
		\scalarl{w_{0}}{v} - dI(\psi_{0})[ da(\eta_{0}[\xi] w_{0}] \\
		&= a(\eta_{0}) \omega(\psi_{0}) \scalarl{w_{0}}{v} +
		\omega(\psi_{0}) \scalarl{\eta_{0}}{\xi} \\
		&= \omega(\psi_{0}) \scalarl{\psi_{0}}{h}
	\end{align*}
	that is
	\begin{equation*}
		dI(\psi_{0})[h] = \omega(\psi_{0})
		\scalarl{\psi_{0}}{h} \qquad \text{for all } h \in X,
	\end{equation*}
	which shows that $\psi_{0}$ is a critical point for $I(\psi)$ under
	the constraint $\norml{\psi} = 1$. The Lagrange multiplier
	$\omega(\psi_{0}) = a(\eta_{0})^{-1} dI(\psi_{0})[w_{0}]$ is such that
	\begin{align*}
		\omega(\psi_{0}) &= a(\eta(w_{0}))^{-1} dI(\psi_{0})[w_{0}]
		\geq \normh{w_{0}}^{2} - s\acc a(\eta(w_{0}))^{-1}
		\normh{\psi_{0}}\normh{w_{0}} \\
		&\geq \normh{w_{0}}^{2} - \frac{s\acc}{2}\left(
		a(\eta(w_{0}))^{-2} \normh{\psi_{0}}^{2} +
		\normh{w_{0}}^{2}\right) \\
		&\geq \normh{w_{0}}^{2} - \frac{s\acc}{2}\left(
		\normh{w_{0}}^{2} + a(\eta(w_{0}))^{-2} \normh{\eta_{0}}^{2} +
		\normh{w_{0}}^{2}\right) \\
		&\geq \normh{w_{0}}^{2} - \frac{3s\acc}{2}
		\normh{w_{0}}^{2}
	\end{align*}
	and 
	\begin{equation*}
		\omega(\psi_{0}) = dI(\psi_{0})[\psi_{0}] \leq 2 I(\psi_{0}).
	\end{equation*}
	
	Suppose now that $\psi_{0} \in \Sigma$ satiesfies
	\eqref{eq:equazionemoltiplicatore} for some $\tilde{\omega}$. Let
	$w_{0} = \norml{\Lambda_{+}\psi_{0}}^{-1} \Lambda_{+} \psi_{0}$
	and $\eta_{0} = \Lambda_{-}\psi_{0}$. Then we deduce from
	\eqref{eq:derivataJ} that for all $\xi \in X_{-}$
	\begin{equation*}
		dJ_{w_{0}}(\eta_{0})[\xi] =
		dI(\psi_{0})[da(\eta_{0})[\xi]w_{0} + \xi] = \tilde{\omega}
		\scalarl{\psi_{0}}{da(\eta_{0})[\xi]w_{0} + \xi} = 0,
	\end{equation*}
	and $\eta_{0}$ is a critical point of $J_{w_{0}}$. From lemma
	\ref{lem:stime_max} we know that $\eta_{0}$ is a local maximum and
	arguing as in the proof of proposition \ref{prop:massimo} we
	deduce that $\eta_{0} = \eta(w_{0})$ and $\mathcal{E}(w_{0}) =
	J_{\omega_{0}}(\eta_{0})$. We also have that
	\begin{equation*}
		\tilde{\omega} = dI(\psi_{0})[\psi_{0}] = \omega(\psi_{0})
	\end{equation*}
	We then deduce from 
	\eqref{eq:derivataE} that 
	\begin{align*}
		d\mathcal{E}(w_{0}) &= a(\gamma(w_{0}))dI(\psi_{0})[v] -
		a(\gamma(w_{0}))^{2} \omega(\psi_{0}) \scalarl{w_{0}}{v} \\
		&= a(\gamma(w_{0})) \tilde{\omega} \scalarl{\psi_{0}}{v} -
		a(\gamma(w_{0}))^{2} \omega(\psi_{w_{0}}) \scalarl{w_{0}}{v}
		\\
		&= a(\gamma(w_{0})) \omega(\psi_{0})
		\scalarl{a(\gamma(w_{0}))w_{0} + \xi}{v} -
		a(\gamma(w_{0}))^{2} \omega(\psi_{0}) \scalar{w_{0}}{v} = 0
	\end{align*}
\end{proof}

From now on we will make explicit the dependence of $I$, $J$ and
$\mathcal{E}$ on $s > 0$ writing $I_{s}$, $J_{s}$ and
$\mathcal{E}_{s}$ and introduce the following minimiziation problem:
\begin{multline*}
	e(s) = \inf_{w \in \Sigma_{+}} \mathcal{E}_{s}(w) \\
	= \inf_{w \in \Sigma_{+}} \left\{ \frac{1}{2}
	a(\eta_{s}(w))^{2}\normh{w}^{2} - \frac{1}{2}
	\normh{\eta_{s}(w)}^{2} - \frac{s}{4} \int
	\frac{\scalar{\psi_{w}}{\beta \psi_{w}}(x) \scalar{\psi_{w}}{\beta
	\psi_{w}}(y)}{\abs{x - y}} \right\}
\end{multline*}
and let $E(s) = se(s)$.

The next lemma is allows us to recover enough compactness (via
the concentration compactness lemma \cite{Lions_1984, Lions_1984-1})
in order to prove our main result, see also \cite[Lemma
4.2]{Nolasco_2021}.
\begin{lem}
	\label{lem:stimAe}
	For all $s \in (0,\frac{1}{8\pi}]$ we have that $0 < e(s) <
	\frac{m}{2}$.
\end{lem}

\begin{proof}
	From lemma \ref{lem:stime_max} we have that $e(s) \geq
	\frac{1}{4}(2 - s\acc)m \geq \frac{1}{4}(2 - \acc)m>
	0$.
	
	Using lemma \ref{lem:stimaMalefica} we deduce that
	\begin{align*}
		\mathcal{E}_{s}(w) &= I_{s}(\psi_{w}) \\
		&\leq \frac{m}{2} + \frac{1}{2} (1 + 8 s \acc)
		(\normh{w}^{2} - m\norml{w}^{2}) - \frac{s}{4}\int
		\frac{\scalar{w}{\beta w}(x) \scalar{w}{\beta w}(y)}{\abs{x -
		y}}
	\end{align*}
	
	Fix $w_{1} \in H^{1}(\mathbb{R}^{3}, \mathbb{C}^{2})$ such that
	$\norml{w_{1}} = 1$, $w = \left(\begin{smallmatrix} w_{1} \\ 0
	\end{smallmatrix}\right)$ and let $w_{\epsilon}(x) =
	\epsilon^{3/2} w(\epsilon x)$. We have that
	\begin{equation*}
		\norml{w_{\epsilon}}^{2} =  \norml{w}^{2} = 1,
	\end{equation*}
	and 
	\begin{equation*}
		\hat{w}_{\epsilon}(q) = \frac{1}{\epsilon^{3/2}} 
		w(\epsilon^{-1}p)
	\end{equation*}
	so that
	\begin{multline*}
		\normh{w_{\epsilon}}^{2} - m\norml{w_{\epsilon}}^{2} = \int
		\left(\sqrt{\abs{q}^{2} + m^{2}} - m\right)
		\abs{\hat{w}_{\epsilon}(q)}^{2} \\
		= \int \left(\sqrt{\epsilon^{2} \abs{q}^{2} + m^{2}} -
		m\right) \abs{\hat{w}_{1}(q)}^{2} \leq \frac{\epsilon^{2}}{2m}
		\int \abs{q}^{2} \abs{\hat{w}_{1}(q)}^{2}
	\end{multline*}
	and $\normh{w_{\epsilon}}^{2} \leq m + C\epsilon^{2}$. We then
	observe that
	\begin{align*}
		\normh{w_{\epsilon} - &\Lambda_{+} w_{\epsilon}}^{2} =
		\normh{\Lambda_{-} w_{\epsilon}}^{2} \\
		&= \int \sqrt{\epsilon^{2} \abs{p}^{2} + m^{2}}
		\labs{\frac{1}{2} \left[ \matr{I} -
		\frac{m\beta}{\sqrt{\epsilon^{2} \abs{p}^{2} + m^{2}}} -
		\frac{\epsilon \malf \cdot p}{\sqrt{\epsilon^{2} \abs{p}^{2} +
		m^{2}}}\right] \begin{pmatrix} \hat{w}_{1}(p) \\
		0 \end{pmatrix} }^{2} \\
		&= \frac{1}{4} \int \sqrt{\epsilon^{2} \abs{p}^{2} + m^{2}}
		\labs{\begin{pmatrix} \frac{\sqrt{\epsilon\abs{p}^{2} + m^{2}}
		- m}{\sqrt{\epsilon\abs{p}^{2} + m^{2}} }\hat{w}_{1}(p) \\
		\frac{\epsilon \boldsymbol{\sigma} \cdot
		p}{\sqrt{\epsilon\abs{p}^{2} + m^{2}}}
		\hat{w}_{1}(p)\end{pmatrix} }^{2}\\
		&= \frac{1}{2}\int \left(\sqrt{\epsilon^{2} \abs{p}^{2} +
		m^{2}} - m\right) \abs{\hat{w}_{1}(p)}^{2} \leq
		\frac{\epsilon^{2}}{4m} \int \abs{p}^{2}
		\abs{\hat{w}_{1}(p)}^{2}
	\end{align*}
	and also
	\begin{multline*}
		\labs{1 - \norml{\Lambda_{+}w_{\epsilon}}} =
		\labs{\norml{w_{\epsilon}} - \norml{\Lambda_{+}w_{\epsilon}}}
		\leq \norml{w_{\epsilon} - \Lambda_{+} w_{\epsilon}} \\
		= \norml{\Lambda_{-} w_{\epsilon}} \leq
		\frac{\epsilon}{2\sqrt{m}} \left(\int \abs{p}^{2}
		\abs{\hat{w}_{1}(p)}^{2}\right)^{1/2}.
	\end{multline*}
	We deduce from this that for $\epsilon > 0$ small enough 
	$\norml{\Lambda_{+}w_{\epsilon}} > \frac{1}{2}$.
	
	Let  
	\begin{equation*}
		\varphi_{\epsilon}(x) = \norml{\Lambda_{+} w_{\epsilon}}^{-1}
		\Lambda_{+} w_{\epsilon}(x).
	\end{equation*}
	We have that
	\begin{equation*}
		\normh{\varphi_{\epsilon}} \leq \norml{\Lambda_{+}
		w_{\epsilon}}^{-1} \normh{w_{\epsilon}} \leq \sqrt{m} + C\epsilon
	\end{equation*}
	and
	\begin{multline*}
		\normh{w_{\epsilon} - \varphi_{\epsilon}} \leq
		\normh{w_{\epsilon} - \Lambda_{+}w_{\epsilon}} +
		\normh{(1 - \norml{\Lambda_{+}w_{\epsilon}})
		\varphi_{\epsilon} } \\
		\leq \frac{\epsilon}{2\sqrt{m}} (2 +
		\normh{\varphi_{\epsilon}}) \left(\int \abs{p}^{2}
		\abs{\hat{w}_{1}(p)}^{2}\right)^{1/2}
	\end{multline*}
	and also
	\begin{equation*}
		\norml{\varphi_{\epsilon} - w_{\epsilon}} =
		\frac{1}{\norml{\Lambda_{+}w_{\epsilon}}}\norml{w_{\epsilon} -
		\Lambda_{+} w_{\epsilon}} \leq \frac{\epsilon}{\sqrt{m}}
		\left(\int \abs{p}^{2} \abs{\hat{w}_{1}(p)}^{2}\right)^{1/2}
	\end{equation*}
	and we can estimate
	\begin{equation*}
		\mathcal{E}_{s}(\varphi_{\epsilon}) \leq \frac{m}{2} +
		\frac{1}{2} (1 + 8 s \acc)
		(\normh{\varphi_{\epsilon}}^{2} -
		m\norml{\varphi_{\epsilon}}^{2}) - \frac{s}{4}\int
		\frac{\scalar{\varphi_{\epsilon}}{\beta \varphi_{\epsilon}}(x)
		\scalar{\varphi_{\epsilon}}{\beta
		\varphi_{\epsilon}}(y)}{\abs{x - y}}.
	\end{equation*}
	We have that
	\begin{align*}
		\normh{\varphi_{\epsilon}}^{2} -
		m\norml{\varphi_{\epsilon}}^{2} &= \int \left(\sqrt{\abs{q}^{2}
		+ m^{2}} - m\right) \abs{\hat{\varphi}_{\epsilon}(q)}^{2} \\
		&\leq 2 \int \left(\sqrt{\abs{q}^{2} + m^{2}} - m\right)
		\abs{\hat{\varphi}_{\epsilon}(q) - \hat{w}_{\epsilon}(q)}^{2}
		\\
		&\qquad + 2 \int \left(\sqrt{\abs{q}^{2} + m^{2}} - m\right)
		\abs{\hat{ w}_{\epsilon}(q)}^{2} \\
		&= 2(\normh{\varphi_{\epsilon} - w_{\epsilon}}^{2} -
		m\norml{\varphi_{\epsilon} - w_{\epsilon}}^{2}) +
		2(\normh{w_{\epsilon}}^{2} - m\norml{w_{\epsilon}}^{2}) \\
		&\leq \frac{\epsilon^{2}}{2m} (3 +
		\normh{\varphi_{\epsilon}})^{2} \norml{\nabla w_{1}}^{2}.
	\end{align*}
	We have that
	\begin{align*}
		Q(\varphi_{\epsilon}) - &Q(w_{\epsilon}) =
		Q((\varphi_{\epsilon} - w_{\epsilon}) + w_{\epsilon})) -
		Q(w_{\epsilon})\\
		&= Q(\varphi_{\epsilon} - w_{\epsilon}) + 4\int
		\frac{\scalar{\varphi_{\epsilon} - w_{\epsilon}}{\beta
		w_{\epsilon}}(x) \scalar{w_{\epsilon}}{\beta
		w_{\epsilon}}(x)}{\abs{x-y}} \\
		&\quad + 3\int \frac{\scalar{\varphi_{\epsilon} -
		w_{\epsilon}}{\beta w_{\epsilon}}(x)
		\scalar{\varphi_{\epsilon} - w_{\epsilon}}{\beta
		w_{\epsilon}}(x)}{\abs{x-y}} \\
		&\quad + 3\int \frac{\scalar{\varphi_{\epsilon} -
		w_{\epsilon}}{\beta (\varphi_{\epsilon} - w_{\epsilon})}(x)
		\scalar{w_{\epsilon}}{\beta w_{\epsilon}}(x)}{\abs{x-y}} \\
		&\quad + 4\int \frac{\scalar{\varphi_{\epsilon} -
		w_{\epsilon}}{\beta (\varphi_{\epsilon} - w_{\epsilon})}(x)
		\scalar{\varphi_{\epsilon} - w_{\epsilon}}{\beta
		w_{\epsilon}}(x)}{\abs{x-y}}.
		\\
		&\geq  - 4\acc \norml{\varphi_{\epsilon} -
		w_{\epsilon}}\norml{(-\Delta)^{1/4}\abs{w_{\epsilon}}}^{2} 
		- 3\acc \norml{\varphi_{\epsilon} -
		w_{\epsilon}}^{2}\norml{(-\Delta)^{1/4}\abs{w_{\epsilon}}}^{2} \\
		&\qquad- 4 \acc \norml{\varphi_{\epsilon} -
		w_{\epsilon}} \normh{\varphi_{\epsilon} - w_{\epsilon}}^{2}.
	\end{align*}
	Since
	\begin{equation*}
		\norml{(-\Delta)^{1/4}\abs{w_{\epsilon}}}^{2} = \int \abs{p}
		\abs{\hat{w}_{\epsilon}}^{2} = \epsilon \int \abs{p} \abs{w_{1}}^{2}
	\end{equation*}
	we have
	\begin{equation*}
		Q(\varphi_{\epsilon}) \geq Q(w_{\epsilon}) - c\epsilon^{2}
		\geq \epsilon Q(w) - c\epsilon^{2}
	\end{equation*}
	we therefore deduce that
	\begin{equation*}
		\mathcal{E}_{s}(w_{\epsilon}) \leq \frac{m}{2} + \frac{2
		\epsilon^{2}}{m}(1 + 8 s \acc)\norml{\nabla w}^{2} -
		\epsilon \frac{s}{4}\int \frac{\scalar{w}{\beta w}(x)
		\scalar{w}{\beta w}(y)}{\abs{x - y}} + c\epsilon^{2}.
	\end{equation*}

	Since $Q(w) > 0$ we deduce that 
	\begin{equation*}
		\mathcal{E}_{s}(w_{\epsilon}) < \frac{m}{2}
	\end{equation*}
	for $\epsilon$ small enough and for all $s \in
	(0,\frac{1}{8\pi})$, and hence $e(s) < \frac{m}{2}$ for all $s \in
	(0,\frac{1}{8\pi})$.
\end{proof}

\begin{prop}
	\label{prop:stimee}
	For all $\theta > 1$ and $s \in (0,\frac{1}{8\pi})$ such that
	$\theta s \in (0,\frac{1}{8\pi})$ we have that
	\begin{equation*}
		\label{eq:stimee}
		e(\theta s) < e(s).
	\end{equation*}
\end{prop}

\begin{proof}
	Let $\theta > 1$ and $s \in (0,\frac{1}{8\pi})$ such that $\theta
	s \in (0,\frac{1}{8\pi})$. Take $w \in \Sigma_{+}$ and let
	$\eta_{s}(w) \in B_{1}$ the function whose existence follows from
	proposition \ref{prop:massimo}. Since it follows from
	\eqref{eq:stimeeta} that
	\begin{equation*}
		\normh{a(\eta_{\theta s}(w)) w}^{2} - \normh{\eta_{\theta
		s}(w)}^{2} - m \geq 0
	\end{equation*}
	we have that
	\begin{align*}
		\theta&(\mathcal{E}_{\theta s}(w) - \frac{m}{2}) \\
		&= \theta\left(\frac{1}{2}\normh{a(\eta_{\theta s}(w)) w}^{2} -
		\frac{1}{2}\normh{\eta_{\theta s}(w)}^{2} - \frac{m}{2} -
		\frac{\theta s}{4} \int \frac{\scalar{\psi_{1}}{\beta
		}(x) \scalar{}{\beta }(y)}{\abs{x -
		y}} \right) \\
		&\leq \theta^{2}\left(\frac{1}{2}\normh{a(\eta_{\theta s}(w))
		w}^{2} - \frac{1}{2}\normh{\eta_{\theta s}(w)}^{2} -
		\frac{m}{2} - \frac{s}{4} \int \frac{\scalar{\psi_{1}}{\beta
		\psi_{1}}(x) \scalar{\psi_{1}}{\beta \psi_{1}}(y)}{\abs{x -
		y}} \right) \\
		&\leq \theta^{2}\left(\frac{1}{2}\normh{a(\eta_{s}(w))
		w}^{2} - \frac{1}{2}\normh{\eta_{s}(w)}^{2} -
		\frac{m}{2} - \frac{s}{4} \int \frac{\scalar{\psi_{2}}{\beta
		\psi_{2}}(x) \scalar{\psi_{2}}{\beta \psi_{2}}(y)}{\abs{x -
		y}} \right) \\
		&= \theta^{2}(\mathcal{E}_{s}(w) - \frac{m}{2})
	\end{align*}
	(here $\psi_{1} = a(\eta_{\theta s}(w)) w + \eta_{\theta s}(w))$
	and $\psi_{2} = a(\eta_{s}(w)) w + \eta_{s}(w)$.) We know that
	$e(s) < \frac{m}{2}$ and hence
	\begin{equation*}
		\theta(e(\theta s) - \frac{m}{2}) \leq \theta^{2}(e(s) -
		\frac{m}{2}) < \theta(e(s) - \frac{m}{2})
	\end{equation*}
	from which we deduce that $e(\theta s) < e(s)$.
\end{proof}

\begin{proof}[Proof of theorem \ref{thm:main}]

By Ekeland's variational principle, there exists a sequence $w_{n} \in
\Sigma_{+}$ such that 
\begin{equation*}
	\mathcal{E}_{s}(w_{n}) \to e(s), \qquad \sup_{v \in \Sigma_{+}}
	\labs{d\mathcal{E}_{s}(w_{n})[v]} \to 0.
\end{equation*}
From $\mathcal{E}_{s}(w_{n}) \to e(s)$ we deduce from lemma
\ref{lem:stime_max} that $\normh{w_{n}} \leq \frac{4e(s)}{2 -
s\acc} + o(1)$ so that the sequence $w_{n}$ is bounded. Follows
from \ref{prop:massimo} that also $\eta_{n} = \eta(w_{n})$ and
$\psi_{n} = a(\eta_{n}) w_{n} + \eta_{n}$ are bounded in $X$. Letting
$\omega_{n} = a(\eta_{n})^{-1}dI(\psi_{n})[w_{n}]$ we have that
\begin{equation*}
	dI_{s}(\psi_{n})[h] - \omega_{n} \scalarl{\psi_{n}}{h} = 0 \qquad
	\text{for all } h \in X.
\end{equation*}

We can assume that (up to a subsequence) $\psi_{n} \tow \psi$ in $X$
and that $\omega_{n} \to \omega$. Then we have that for all $h \in X$
\begin{multline*}
	dI_{s}(\psi_{n})[h] - \omega_{n} \scalarl{\psi_{n}}{h} \\
	= \scalarh{\psi_{n}}{\Lambda_{+} h} -
	\scalarh{\psi_{n}}{\Lambda_{-} h} - s \int
	\frac{\scalar{\psi_{n}}{\beta \psi_{n}}(x) \scalar{\psi_{n}}{\beta
	h}(y)}{\abs{x - y}} - \omega_{n} \scalarl{\psi_{n}}{h} \\
	\to 0
\end{multline*}
since, by \eqref{eq:convoluzione_tende_0}, we have that
\begin{equation*}
	\int \frac{\scalar{\psi_{n}}{\beta \psi_{n}}(x) \scalar{\psi_{n} -
	\psi}{\beta h}(y)}{\abs{x - y}} \to 0.
\end{equation*}
As a consequence we have that
\begin{equation*}
	dI_{s}(\psi)[h] - \omega \scalarl{\psi}{h} = 0 \qquad \text{for
	all } h \in X.
\end{equation*}
The weak convergence does not, however, preserve the $L^{2}$ norm, so 
we only know that $\norml{\psi} \leq \norml{\psi_{n}} = 1$ (it could 
even be that $\psi = 0$). 

To conclude we will now apply the concentration-compactness principle,
see \cite{Lions_1984, Lions_1984-1}. First of all let us show that no
vanishining occurs. By contradiction, assume that
\begin{equation*}
	\limsup_{n \to +\infty} \sup_{y \in \mathbb{R}^{3}} \int_{B(y,1)} 
	\abs{\psi_{n}}^{2} = 0.
\end{equation*}
Then we know, see \cite{Lions_1984} or \cite[Lemma 1.21]{Willem96},
that $\psi_{n} \to 0$ in $L^{p}(\mathbb{R}^{3})$ for $2 < p < 6$.

Since 
\begin{equation*}
	\int \frac{\scalar{\psi_{n}}{\beta \psi_{n}(x)}
	\scalar{\psi_{n}}{\beta \psi_{n}(y)}}{\abs{x - y}} \leq \int
	\frac{\abs{\psi_{n}(x)}^{2} \abs{\psi_{n}(y)}^{2}}{\abs{x - y}}
	\leq C\norml[\frac{12}{5}]{\psi_{n}}^{4} \to 0
\end{equation*}
we deduce, using \eqref{eq:stimeeta}, \eqref{eq:stimeMoltiplicatore}
and lemma \ref{lem:stimAe} that
\begin{align*}
	0 &= dI_{s}(\psi_{n})[\psi_{n}] - \omega_{n} \norml{\psi_{n}}^{2} \\
	&= \normh{a(\eta_{n})w_{n}}^{2} - \normh{\eta_{n}}^{2} - \omega_{n}
	\norml{\psi_{n}}^{2} - s \int \frac{\scalar{\psi_{n}}{\beta
	\psi_{n}(x)} \scalar{\psi_{n}}{\beta \psi_{n}(y)}}{\abs{x - y}} \\
	&= \normh{a(\eta_{n})w_{n}}^{2} - \normh{\eta_{n}}^{2} -
	m\norml{\psi_{n}}^{2} + (m - \omega_{n}) + o(1) \\
	&\geq (m - \omega_{n}) + o(1) > 0
\end{align*}
for $n$ large enough, a contradiction which shows that vanishing does 
not occur.

Then we know from the concentration compactness principle, that there
exists $p \geq 1$ functions $\phi_{1}$, \ldots, $\phi_{p} \in X$,
critical points for $I_{s}$ under the constraint $\norml{\psi}^{2} =
\mu_{i}$ (hence satisfying \eqref{eq:equazionemoltiplicatore} with
$\omega = \lim_{n}\omega_{n} > 0$) and $p$ sequences of points $x_{i,
n} \in \mathbb{R}^{3}$, $i = 1$, \ldots, $p$ such that $\abs{x_{i,n} -
x_{j,n}} \to +\infty$ for all $i \neq j$ as $n \to +\infty$ and
\begin{equation*}
	\normh{\psi_{n} - \sum_{i = i}^{^{p}} \phi_{i}(\cdot - x_{i,n})} 
	\to 0 \quad \text{as } n \to +\infty.
\end{equation*}
From this follows also that $\norml{\psi_{n}}^{2} = 1 = \sum_{i =
1}^{p} \mu_{i}$.

We then observe that
\begin{align*}
	\normh{\Lambda_{+}\psi_{n}}^{2} - &\normh{\Lambda_{-}\psi_{n}}^{2}
	= \scalarh{\psi_{n}}{\Lambda_{+}\psi_{n} - \Lambda_{-}\psi_{n}}
	\\
	&=\scalarh{\psi_{n} - \sum_{i = i}^{^{p}} \phi_{i}(\cdot -
	x_{i,n})}{\Lambda_{+}\psi_{n} - \Lambda_{-}\psi_{n}} \\
	&\qquad + \sum_{i = i}^{^{p}} \scalarh{\phi_{i}(\cdot -
	x_{i,n})}{\Lambda_{+}\psi_{n} - \Lambda_{-}\psi_{n}} \\
	&= \sum_{i = i}^{^{p}} \left(\scalarh{\Lambda_{+}\phi_{i}(\cdot -
	x_{i,n})}{\psi_{n}} - \scalarh{\Lambda_{-}\phi_{i}(\cdot -
	x_{i,n})}{\psi_{n}}\right) + o(1) \\
	&= \sum_{i = i}^{^{p}} \left(\normh{\Lambda_{+}\phi_{i}}^{2} -
	\normh{\Lambda_{-}\phi_{i}}^{2} \right) + o(1)
\end{align*}
and also
\begin{equation*}
	\int \frac{\scalar{\psi_{n}}{\beta \psi_{n}(x)}
	\scalar{\psi_{n}}{\beta \psi_{n}(y)}}{\abs{x - y}} =
	\sum_{i=1}^{p} \int \frac{\scalar{\phi_{i}}{\beta \phi_{i}(x)}
	\scalar{\phi_{i}}{\beta \phi_{i}(y)}}{\abs{x - y}} + o(1).
\end{equation*}
Finally we have that
\begin{equation}
	\label{eq:Isispezza}
	e(s) = I_{s}(\psi_{n})+ o(1) = \sum_{i=1}^{p} I_{s}(\phi_{i}) +
	o(1)
\end{equation}

Let, for $i = 1, \ldots,n$, $\tilde{\psi}_{i} = \norml{\phi_{i}}^{-1}
\phi_{i} = \mu_{i}^{-1/2} \phi_{i} \in \Sigma$. We have that
\begin{equation*}
	I_{s}(\phi_{i}) = I_{s}(\sqrt{\mu_{i}} \tilde{\psi}_{i}) = 
	\mu_{i} I_{s\mu_{i}}(\tilde{\psi}_{i}) 
\end{equation*}
and
\begin{equation*}
	0 = dI_{s}(\phi_{i})[h] - \omega \scalarl{\phi_{i}}{h} =
	\sqrt{\mu_{i}} \left(dI_{s\mu_{i}}(\tilde{\psi}_{i})[h] - \omega
	\scalarl{\tilde{\psi}_{i}}{h}\right) \qquad \text{for all } h \in 
	X,
\end{equation*}
Follows from proposition \ref{prop:punticriticiE} that $\tilde{w}_{i}
= \norml{\Lambda_{+}\tilde{\psi}_{i}}^{-1} \Lambda_{+}\tilde{\psi}_{i}
\in \Sigma_{+}$ is a critical point for $\mathcal{E}_{s\mu_{i}}$ and
$\mathcal{E}_{s\mu_{i}}(\tilde{w}_{i}) =
I_{s\mu_{i}}(\tilde{\psi}_{i})$.

Since
\begin{equation*}
	\mathcal{E}_{s\mu_{i}}(\tilde{w}_{i}) \geq e(s\mu_{i}) 
\end{equation*}
we deduce from \eqref{eq:stimee} that
\begin{equation*}
	e(s) = \sum_{i=1}^{p} I_{s}(\phi_{i}) = \sum_{i=1}^{p} \mu_{i}
	I_{s\mu_{i}}(\tilde{\psi}_{i}) \geq \sum_{i=1}^{p} \mu_{i}
	e(s\mu_{i}) > \sum_{i=1}^{p} \mu_{i}
	e(\frac{1}{\mu_{i}}s\mu_{i}) = e(s)\sum_{i = 1}^{p} \mu_{i},
\end{equation*}
a contradiction if $p > 1$.
	
Since there is no vanishing or dichotomy, our sequence $\psi_{n}$
converges strongly in $X$ to a critical point $\psi \in X$ of
\eqref{eq:equazionemoltiplicatore} such that $\norml{\psi} = 1$ and
the theorem follows.

\end{proof}

\appendix

\section{A useful lemma}

This lemma is similar to \cite[Lemma 2.9]{Nolasco_2021}. We give here
a slightly different proof.
\begin{lem}
	\label{lem:stimaMalefica}
	For all $\psi = \sqrt{1 - \norml{w}^{2}}w + \eta$, $w \in 
	\Sigma_{+}$, $\eta \in X_{-}$ we have
	\begin{align*}
		\int \frac{\scalar{\psi}{\beta\psi}(x)
		\scalar{\psi}{\beta\psi}(y)}{\abs{x-y}} &\geq \int
		\frac{\scalar{w}{\beta w}(x) \scalar{w}{\beta
		w}(y)}{\abs{x-y}}- 2 \acc \norml{\eta}^{2}
		\normh{w}^{2} \\
		&\qquad -14 \aeta^{2} \acc(\normh{w}^{2} -
		m\norml{w}^{2}) - 18 \acc\normh{\eta}^{2}
	\end{align*}
\end{lem}

\begin{proof}
We have 
\begin{align*}
	&\int \frac{\scalar{\psi}{\beta\psi}(x)
	\scalar{\psi}{\beta\psi}(y)}{\abs{x-y}} = \aeta^{4} \int
	\frac{\scalar{w}{\beta w}(x) \scalar{w}{\beta w}(y)}{\abs{x-y}} \\
	&\qquad + 4\aeta^{3} \int \frac{\scalar{w}{\beta w}(x)
	\scalar{w}{\beta \eta}(y)}{\abs{x-y}} + 3\aeta^{2} \int
	\frac{\scalar{w}{\beta w}(x) \scalar{\eta}{\beta
	\eta}(y)}{\abs{x-y}} \\
	&\qquad + 3\aeta^{2} \int \frac{\scalar{w}{\beta \eta}(x)
	\scalar{w}{\beta \eta}(y)}{\abs{x-y}} + 4\aeta \int
	\frac{\scalar{\eta}{\beta \eta}(x) \scalar{w}{\beta
	\eta}(y)}{\abs{x-y}} \\
	&\qquad + \int \frac{\scalar{\eta}{\beta \eta}(x)
	\scalar{\eta}{\beta \eta}(y)}{\abs{x-y}}\\
	&\geq \aeta^{4} \int \frac{\scalar{w}{\beta w}(x) \scalar{w}{\beta
	w}(y)}{\abs{x-y}} + 4\aeta^{3} \int \frac{\scalar{w}{\beta w}(x)
	\scalar{w}{\beta \eta}(y)}{\abs{x-y}} \\
	&\qquad - 3\aeta^{2}\acc \normh{\eta}^{2} + 4\aeta \int
	\frac{\scalar{\eta}{\beta \eta}(x) \scalar{w}{\beta
	\eta}(y)}{\abs{x-y}} 
\end{align*}

We have
\begin{align*}
	\labs{\int \frac{\scalar{w}{\beta w}(x) \scalar{w}{\beta
	\eta}(y)}{\abs{x-y}} \, dx \, dt } &= (2\pi)^{3/2}
	\sqrt{\frac{2}{\pi}} \labs{ \int
	\frac{\mathcal{F}[\scalar{w}{\beta w}]
	\mathcal{F}[\scalar{w}{\beta \eta}]}{\abs{p}^{2}} \, dp} \\
	&\leq (2\pi)^{3/2} \sqrt{\frac{2}{\pi}}
	\abs{\mathcal{F}[\scalar{w}{\beta w}]}_{\infty} \labs{ \int
	\frac{ \abs{\mathcal{F}[\scalar{w}{\beta \eta}]}}{\abs{p}^{2}} \,
	dp} \\
	&\leq \sqrt{\frac{2}{\pi}} \abs{\scalar{w}{\beta w}}_{1} \labs{
	\int \frac{ \abs{\mathcal{F}[\scalar{w}{\beta
	\eta}]}}{\abs{p}^{2}} \, dp} \\
	&\leq \sqrt{\frac{2}{\pi}} \norml{w}^{2} \labs{ \int \frac{
	\abs{\mathcal{F}[\scalar{w}{\beta \eta}]}}{\abs{p}^{2}} \, dp} \\
	&\leq \sqrt{\frac{2}{\pi}} \int \frac{dp}{\abs{p}^{2}} \left( 
	\frac{1}{(2\pi)^{3/2}} \int \abs{ 
	\scalar{\hat{w}(p-q)}{\beta\hat{\eta}(q)}} \, dq \right)
\end{align*}
and
\begin{multline*}
	\labs{\int \frac{\scalar{\eta}{\beta \eta}(x) \scalar{w}{\beta
	\eta}(y)}{\abs{x-y}} \, dx \, dt } \\
	\leq \norml{\eta}^{2} \sqrt{\frac{2}{\pi}} \int
	\frac{dp}{\abs{p}^{2}} \left( \frac{1}{(2\pi)^{3/2}} \int \abs{
	\scalar{\hat{w}(p-q)}{\beta\hat{\eta}(q)}} \, dq \right)
\end{multline*}
so that
\begin{multline*}
	4\aeta^{3} \int \frac{\scalar{w}{\beta w}(x)
	\scalar{w}{\beta \eta}(y)}{\abs{x-y}}  + 4\aeta \int
	\frac{\scalar{\eta}{\beta \eta}(x) \scalar{w}{\beta
	\eta}(y)}{\abs{x-y}} \\
	\leq 4\aeta \sqrt{\frac{2}{\pi}} \int \frac{dp}{\abs{p}^{2}}
	\left( \frac{1}{(2\pi)^{3/2}} \int \abs{
	\scalar{\hat{w}(p-q)}{\beta\hat{\eta}(q)}} \, dq \right)
\end{multline*}

Since 
\begin{align*}
	\scalar{\hat{w}(p-q)}{\beta\hat{\eta}(q)} &=
	\scalar{\hat{\Lambda}_{+}(p-q) \hat{w}(p-q)}{\beta
	\hat{\Lambda}_{-}(q) \hat{\eta}(q)} \\
	&= \scalar{\hat{w}(p-q)}{\hat{\Lambda}_{+}(p-q) \beta
	\hat{\Lambda}_{-}(q) \hat{\eta}(q)}
\end{align*}
we compute
\begin{align*}
	4\hat{\Lambda}_{+}(p-q) &\beta \hat{\Lambda}_{-}(q) \\
	&=\left(\matr{I} + \frac{m\beta}{\lambda(p-q)} + \frac{\malf \cdot
	(p-q)}{\lambda(p-q)} \right) \beta \left(\matr{I} -
	\frac{m\beta}{\lambda(q)} - \frac{\malf \cdot q}{\lambda(q)} \right) \\
	&= \beta \left( 1 - \frac{m^{2}}{\lambda(q)\lambda(p-q)}\right) -
	\matr{I} \left(\frac{m}{\lambda(q)} -
	\frac{m}{\lambda(p-q)}\right) \\
	&\qquad - \beta \malf \cdot \left(\frac{q}{\lambda(q)} +
	\frac{p-q}{\lambda(p-q)} \right) - \frac{m \malf \cdot (q +
	(p-q))}{\lambda(q) \lambda(p-q)} \\
	&\qquad + \frac{\beta}{\lambda(q) \lambda(p-q)} (\malf \cdot (p-q)
	\malf \cdot q ) \\
	&= \beta \left( 1 - \frac{m^{2}}{\lambda(q)\lambda(p-q)}\right) -
	\matr{I} \left(\frac{m}{\lambda(q)} -
	\frac{m}{\lambda(p-q)}\right) \\
	&\qquad - \beta \malf \cdot \left(\frac{q}{\lambda(q)} +
	\frac{p-q}{\lambda(p-q)} \right) - \frac{m \malf \cdot
	p}{\lambda(q) \lambda(p-q)} 
	+ \frac{\beta \msig \cdot (p-q) \msig\cdot q}{\lambda(q)
	\lambda(p-q)} 
\end{align*}
where 
\begin{equation*}
	\Sigma_{i} = 
	\begin{pmatrix}
		\sigma_{i} & 0 \\
		0 & \sigma_{i}
	\end{pmatrix}
\end{equation*}
We now estimate the different terms. First of all, from 
\begin{equation*}
	m(\abs{p-q} + \abs{q}) \leq \lambda(q) \lambda(p-q) \leq
	\abs{q}\abs{p-q} + m(\abs{q} + \abs{p-q}) + m^{2},
\end{equation*}
and
\begin{equation*}
	\labs{\abs{q}\lambda(p-q) - \abs{p-q}\lambda(q)} \leq m\abs{p}
\end{equation*}
we deduce that
\begin{equation*}
	\labs{\frac{\abs{q}}{\lambda(q)} - \frac{\abs{p-q}}{\lambda(p-q)}}
	= \labs{\frac{\abs{q}\lambda(p-q)
	-\abs{p-q}\lambda(q)}{\lambda(q)\lambda(p-q)}} \leq
	\frac{m\abs{p}}{\lambda(q)\lambda(p-q)}
\end{equation*}
and
\begin{align*}
	\labs{\frac{m}{\lambda(q)} - \frac{m}{\lambda(p-q)}} &= m
	\frac{\abs{\lambda(p-q) - \lambda(q)}}{\lambda(q) \lambda(p-q)} =
	m\frac{\abs{m^{2} + \abs{p-q}^{2} - m^{2} -
	\abs{q}^{2}}}{\lambda(q) \lambda(p-q)(\lambda(q) + \lambda(p
	-q))}\\
	&\leq \frac{\abs{\abs{p-q} - \abs{q}}}{\lambda(q) + \lambda(p -q)}
	\leq \frac{\abs{p}}{\lambda(q) + \lambda(p -q)} \\
	&\leq \frac{\abs{p}}{(\lambda(q) + m)^{1/2}(\lambda(p -q) + 
	m)^{1/2}}.
\end{align*}
Then
\begin{align*}
	\labs{\frac{\lambda(q)\lambda(p-q) -
	m^{2}}{\lambda(q)\lambda(p-q)}} &\leq \frac{\abs{q}\abs{p-q} +
	m(\abs{q} + \abs{p-q})}{\lambda(q)\lambda(p-q)} \\
	&\leq \frac{\abs{q}\abs{p-q}}{\lambda(q)\lambda(p-q)} +
	\frac{m\abs{p}}{\lambda(q)\lambda(p-q)} +
	2\frac{m\abs{p-q}}{\lambda(q)\lambda(p-q)}
\end{align*}
and
\begin{align*}
	\labs{\frac{q}{\lambda(q)} + \frac{p-q}{\lambda(p-q)}} &\leq
	\labs{\frac{\abs{q}}{\lambda(q)} - \frac{\abs{p-q}}{\lambda(p-q)}}
	+ 2 \frac{\abs{p-q}}{\lambda(p-q)} \\
	&\leq \frac{m\abs{p}}{\lambda(q)\lambda(p-q)} + 2
	\frac{\abs{p-q}}{\lambda(p-q)}.
\end{align*}
Since 
\begin{equation*}
	\labs{\frac{\beta \msig \cdot (p-q) \msig\cdot q}{\lambda(q)
	\lambda(p-q)}} \leq \frac{\abs{q} \abs{p-q}}{\lambda(q)
	\lambda(p-q)}
\end{equation*}
we finally have 
\begin{multline*}
	4\abs{\scalar{\hat{w}(p-q)}{\beta \hat{\eta}(q)}} \\
	\leq \Bigl( \frac{3 \abs{q} \abs{p-q} + 3 m\abs{p} + 2 m \abs{p-q}
	}{\lambda(q)\lambda(p-q)} + \frac{2\abs{p-q}}{\lambda(p-q)} \Bigr)
	\abs{\hat{w}(p-q)} \abs{\hat{\eta}(q)} \\+
	\frac{\abs{p}}{(\lambda(q) + m)^{1/2}(\lambda(p -q) + m)^{1/2}}
	\abs{\hat{w}(p-q)} \abs{\hat{\eta}(q)}
\end{multline*}
Let us analyse the different terms:
\begin{align*}
	\sqrt{\frac{2}{\pi}} \int \frac{dp}{\abs{p}^{2}} &\left(
	\frac{1}{(2\pi)^{3/2}} \int \frac{\abs{p-q}
	\abs{\hat{w}(p-q)}}{\lambda(p-q)} \frac{\abs{q}
	\abs{\hat{\eta}(q)}}{\lambda(q)} \, dq \right) \\
	&= \sqrt{\frac{2}{\pi}} \int \frac{1}{\abs{p}^{2}} \mathcal{F}
	\left[ \mathcal{F}^{-1} \left[\frac{\abs{p}
	\abs{\hat{w}(p)}}{\lambda(p)} \right] \mathcal{F}^{-1} \left[
	\frac{\abs{p} \abs{\hat{\eta}(p)}}{\lambda(p)} \right] \right] \, dp
	\\
	&= \int \frac{1}{\abs{x}} \mathcal{F}^{-1} \left[\frac{\abs{p}
	\abs{\hat{w}(p)}}{\lambda(p)} \right] \mathcal{F}^{-1} \left[
	\frac{\abs{p} \abs{\hat{\eta}(p)}}{\lambda(p)} \right] \, dx \\
	&\leq \norml{\frac{1}{\abs{x}^{1/2}} \mathcal{F}^{-1}
	\left[\frac{\abs{p} \abs{\hat{w}(p)}}{\lambda(p)} \right]}
	\norml{\frac{1}{\abs{x}^{1/2}} \mathcal{F}^{-1}
	\left[\frac{\abs{p} \abs{\hat{\eta}(p)}}{\lambda(p)} \right]} \\
	&\leq \acc \norml{(-\Delta)^{1/4} \mathcal{F}^{-1}
	\left[\frac{\abs{p} \abs{\hat{w}(p)}}{\lambda(p)} \right]}
	\norml{(-\Delta)^{1/4} \mathcal{F}^{-1} \left[\frac{\abs{p}
	\abs{\hat{\eta}(p)}}{\lambda(p)} \right]} \\
	&\leq 2 \acc \sqrt{\normh{w}^{2} - m\norml{w}^{2}}
	\sqrt{\normh{\eta}^{2} - m\norml{\eta}^{2}}
\end{align*}
since
\begin{multline*}
	\norml{(-\Delta)^{1/4} \mathcal{F}^{-1} \left[\frac{\abs{p}
	\abs{\hat{w}(p)}}{\lambda(p)} \right]}^{2} = \int
	\frac{\abs{p}^{3} \abs{\hat{w}(p)}^{2}}{\lambda(p)^{2}} \, dp \\
	\leq 2 \int (\sqrt{\abs{p}^{2} + m^{2}} - m) \abs{\hat{w}(p)}^{2}
	\, dp = 2(\normh{w}^{2} - m\norml{w}^{2})
\end{multline*}
and
\begin{align*}
	\sqrt{\frac{2}{\pi}} \int \frac{dp}{\abs{p}^{2}} &\left(
	\frac{1}{(2\pi)^{3/2}} \int \frac{\abs{p-q}
	\abs{\hat{w}(p-q)}}{\lambda(p-q)} \frac{m
	\abs{\hat{w}(q)}}{\lambda(q)} \, dq \right) \\
	&\leq \acc \norml{(-\Delta)^{1/4} \mathcal{F}^{-1}
	\left[\frac{\abs{p} \abs{\hat{w}(p)}}{\lambda(p)} \right]}
	\norml{(-\Delta)^{1/4} \mathcal{F}^{-1} \left[\frac{m
	\abs{\hat{\eta}(p)}}{\lambda(p)} \right]} \\
	&\leq \sqrt{m} \acc \norml{\eta} \sqrt{\normh{w}^{2} -
	m\norml{w}^{2}}
\end{align*}
since
\begin{equation*}
	\norml{(-\Delta)^{1/4} \mathcal{F}^{-1} \left[\frac{m
	\abs{\hat{\eta}(p)}}{\lambda(p)} \right]}^{2} = \int
	\frac{m^{2}\abs{p} \abs{\hat{\eta}(p)}^{2}}{\lambda(p)^{2}} \, dp
	\leq \frac{m}{2} \norml{\eta}^{2}.
\end{equation*}
We also have that
\begin{align*}
	\sqrt{\frac{2}{\pi}} \int \frac{dp}{\abs{p}^{2}} &\left(
	\frac{1}{(2\pi)^{3/2}} \int \abs{p} \frac{
	\abs{\hat{w}(p-q)}}{(\lambda(p-q) + m)^{1/2}} \frac{
	\abs{\hat{\eta}(q)}}{(\lambda(q) + m)^{1/2}} \, dq \right) \\
	&= \sqrt{\frac{2}{\pi}} \int \frac{1}{\abs{p}} \mathcal{F} \left[
	\mathcal{F}^{-1} \left[\frac{ \abs{\hat{w}(p)}}{(\lambda(p) +
	m)^{1/2}} \right] \mathcal{F}^{-1} \left[ \frac{\abs{p}
	\abs{\hat{\eta}(p)}}{(\lambda(p) + m)^{1/2}} \right] \right] \, dp
	\\
	&= \frac{2}{\pi} \int \frac{1}{\abs{x}^{2}} \mathcal{F}^{-1}
	\left[\frac{ \abs{\hat{w}(p)}}{(\lambda(p) + m)^{1/2}} \right]
	\mathcal{F}^{-1} \left[ \frac{\abs{\hat{\eta}(p)}}{(\lambda(p) +
	m)^{1/2}} \right] \, dx \\
	&\leq \frac{2}{\pi} \norml{\frac{1}{\abs{x}} \mathcal{F}^{-1}
	\left[\frac{\abs{\hat{w}(p)}}{(\lambda(p) + m)^{1/2}} \right]}
	\norml{\frac{1}{\abs{x}^{1/2}} \mathcal{F}^{-1}
	\left[\frac{\abs{\hat{\eta}(p)}}{(\lambda(p) + m)^{1/2}} \right]}
	\\
	&\leq \frac{8}{\pi} \norml{\frac{
	\abs{p}\abs{\hat{w}(p)}}{(\lambda(p) + m)^{1/2}}} \norml{\frac{
	\abs{p} \abs{ \hat{\eta}(p)}}{(\lambda(p) + m)^{1/2}} } \\
	&\leq 2 \acc \sqrt{\normh{w}^{2} - m\norml{w}^{2}}
	\sqrt{\normh{\eta}^{2} - m\norml{\eta}^{2}}
\end{align*}
and
\begin{align*}
	\sqrt{\frac{2}{\pi}} \int \frac{dp}{\abs{p}^{2}} &\left(
	\frac{1}{(2\pi)^{3/2}} \int m\abs{p} \frac{
	\abs{\hat{w}(p-q)}}{\lambda(p-q)} \frac{
	\abs{\hat{\eta}(q)}}{\lambda(q)} \, dq \right) \\
	&\leq \frac{8}{\pi} \norml{\frac{ \sqrt{m}
	\abs{p}\abs{\hat{w}(p)}}{\lambda(p)}} \norml{\frac{ \sqrt{m}
	\abs{p} \abs{ \hat{\eta}(p)}}{\lambda(p)} } \\
	&\leq 5 \acc \sqrt{\normh{w}^{2} - m\norml{w}^{2}}
	\sqrt{\normh{\eta}^{2} - m\norml{\eta}^{2}}
\end{align*}
(we have used the fact that $\frac{mp^{2}}{p^{2} + m^{2}} \leq 
\frac{5}{2} (\sqrt{p^{2} +m^{2}} - m)$.)

Finally we have
\begin{align*}
	\sqrt{\frac{2}{\pi}} \int \frac{dp}{\abs{p}^{2}} &\left(
	\frac{1}{(2\pi)^{3/2}} \int \frac{\abs{p-q}
	\abs{\hat{w}(p-q)}\abs{\hat{w}(q)}}{\lambda(p-q)} \, dq \right) \\
	&\leq \acc \norml{(-\Delta)^{1/4} \mathcal{F}^{-1}
	\left[\frac{\abs{p} \abs{\hat{w}(p)}}{\lambda(p)} \right]}
	\norml{(-\Delta)^{1/4} \mathcal{F}^{-1} \left[ \abs{\hat{\eta}(p)}
	\right]} \\
	&\leq \sqrt{2} \acc \normh{\eta} \sqrt{\normh{w}^{2} -
	m\norml{w}^{2}}.
\end{align*}

We now collect the terms:
\begin{align*}
	4\aeta \sqrt{\frac{2}{\pi}} &\int \frac{dp}{\abs{p}^{2}}
	\left( \frac{1}{(2\pi)^{3/2}} \int \abs{
	\scalar{\hat{w}(p-q)}{\beta\hat{\eta}(q)}} \, dq \right) \\
	&\leq 23\acc \aeta \sqrt{\normh{w}^{2} - m\norml{w}^{2}}
	\sqrt{\normh{\eta}^{2} - m\norml{\eta}^{2}} \\
	&\qquad + 2 \sqrt{m} \acc \aeta \norml{\eta}
	\sqrt{\normh{w}^{2} - m\norml{w}^{2}} \\
	&\qquad + 2\sqrt{2}\acc\aeta \normh{\eta}
	\sqrt{\normh{w}^{2} - m\norml{w}^{2}} \\
	&\leq 14 \aeta^{2} \acc(\normh{w}^{2} - m\norml{w}^{2}) + 15
	\acc\normh{\eta}^{2}
\end{align*}
to deduce that
\begin{align*}
	\int \frac{\scalar{\psi}{\beta\psi}(x)
	\scalar{\psi}{\beta\psi}(y)}{\abs{x-y}} &- \int
	\frac{\scalar{w}{\beta w}(x) \scalar{w}{\beta w}(y)}{\abs{x-y}} \\
	&\geq - 2 \acc \norml{\eta}^{2} \normh{w}^{2} -
	3\aeta^{2}\acc \normh{\eta}^{2} \\
	&\quad - 4\aeta \sqrt{\frac{2}{\pi}} \int \frac{dp}{\abs{p}^{2}}
	\left( \frac{1}{(2\pi)^{3/2}} \int \abs{
	\scalar{\hat{w}(p-q)}{\beta\hat{\eta}(q)}} \, dq \right) \\
	&\geq -14 \aeta^{2} \acc(\normh{w}^{2} - m\norml{w}^{2}) - 2
	\acc \norml{\eta}^{2} \normh{w}^{2} - 18
	\acc\normh{\eta}^{2}
\end{align*}
\end{proof}


\providecommand{\href}[2]{#2}
\renewcommand{\MR}[1]{}

\end{document}